\documentclass[smallextended,referee,envcountsect,]{svjour3}
\smartqed
\usepackage{graphicx}
\usepackage{fix-cm}
\providecommand{\norm}[1]{\lVert#1\rVert}
\usepackage{amssymb}
\usepackage{amsmath}
\usepackage{todonotes}
\usepackage{mathtools}
\DeclareMathOperator*{\argmin}{arg\,min}
\newtheorem{assumption}{Assumption}
\usepackage[colorlinks=true, linkcolor=blue, citecolor=red, urlcolor=magenta]{hyperref}
\usepackage{lineno}
\usepackage{adjustbox}
\usepackage{subcaption}
\date{Received: date / Accepted: date}
\usepackage{hyperref}

\usepackage[indLines=false,noEnd=false]{algpseudocodex}
\usepackage[kerning, spacing]{microtype}
\usepackage{multirow}

\usepackage{float}
\floatstyle{ruled}
\floatname{algorithm}{Algorithm}
\newfloat{algorithm}{tbp!}{loa}
\numberwithin{algorithm}{section}
\usepackage{comment}
\begin{document}
\title{Nonmonotone higher-order Taylor approximation methods for composite problems %$^*$ %\thanks{}
%about the article that should go on the front page should be
%placed here. General acknowledgments should be placed at the end of the article.}
%}
%\subtitle{Do you have a subtitle?\\ If so, write it here}
}

\titlerunning{Nonmonotone higher-order Taylor approximation...}        % if too long for running head

\author{Yassine Nabou $^*$ %\textsuperscript{\textsection} %etc.
}
\institute{$^*$Department of Mathematics and Statistics, University of Helsinki, Finland.\\
              \email{yassine.nabou@helsinki.fi}\\
This work has been supported by the Academy of Finland grant 345486.\\
}

\maketitle

\begin{abstract}

We study composite optimization problems in which the smooth part of the objective function is \( p \)-times continuously differentiable, where \( p \geq 1 \) is an integer. Higher-order methods are known to be effective for solving such problems, as they speed up convergence rates. These methods often require, or implicitly ensure, a monotonic decrease in the objective function across iterations. Maintaining this monotonicity typically requires that the \( p \)-th derivative of the smooth part of the objective function is globally Lipschitz or that the generated iterates remain bounded. In this paper, we propose nonmonotone higher-order Taylor approximation (NHOTA) method for composite problems. Our method achieves the same nice
global and rate of convergence properties as traditional higher-order methods while eliminating the need for global Lipschitz continuity assumptions, strict descent condition, or explicit boundedness of the iterates. Specifically, for nonconvex composite problems, we derive global convergence rate to a stationary point of order \( \mathcal{O}(k^{-\frac{p}{p+1}}) \), where \( k \) is the iteration counter. Moreover, when the objective function satisfies the Kurdyka–Łojasiewicz (KL) property, we obtain improved rates that depend on the KL parameter. Furthermore, for convex composite problems, our method achieves sublinear convergence rate of order \( \mathcal{O}(k^{-p}) \) in function values. Finally, preliminary numerical experiments on nonconvex phase retrieval problems highlight the promising performance of the proposed approach.

\end{abstract}


%%%%%%%%%%%%%%%%%%%%%%%%%%%%%%%%%%%%%%%%
\section{Introduction and motivation}\label{sec:1}
In this paper, we consider the following composite optimization problem:  
\begin{align}  
\label{eq:problem}  
\min_{x \in \mathbb{E}} f(x) := F(x) + h(x),  
\end{align}  
where \( F : \mathbb{E} \to \mathbb{R} \) is a \( p \)-times continuously differentiable function, and \( h: \mathbb{E} \to \bar{\mathbb{R}} \) is a proper, lower semicontinuous, and convex function. Here, \( \mathbb{E} \) is a finite-dimensional real vector space. Despite its simple form, this problem occur frequently in many applications including machine learning \cite{meier2008group}, compressed sensing  \cite{tibshirani1996regression}, matrix factorization \cite{lee1999learning,lin2007projected},  (sparse) inverse covariance selection \cite{hsieh2011sparse,oztoprak2012newton},
blind deconvolution \cite{ayers1988iterative,bertero1998novel}. Further applications can be found
in deep learning \cite{deleu2021structured}, data clustering \cite{ravikumar2011high} and dictionary learning \cite{mairal2009online}.
\medskip

First-Order Methods among the most popular methods for solving problems of the form \eqref{eq:problem}. For example, proximal point (Prox) methods has been extensively studied and are widely recognized as one of the most effective approaches for solving optimization problems of the form \eqref{eq:problem}, where the objective function consists of both a smooth and a non-smooth component. Prox method operates by iteratively refining the solution through a sequence of subproblems, each of which involves the proximal operator associated with the non-smooth part of the objective function. For a given current point \( x_k \), and in the context of solving problems \eqref{eq:problem}, Prox method finds the next iterate by solving
\[
    x_{k+1} = \argmin_{x} \left\{ F(x_k) + \langle \nabla F(x_k), x - x_k \rangle + \frac{1}{2\alpha} \|x - x_k\|^2 + h(x) \right\},
\]
where \( \alpha > 0 \) is the step size that controls the update magnitude. This implies that the new iterate \( x_{k+1} \) is obtained by applying a quadratic approximation to the smooth function \( F \), while the nonsmooth function \( h \) remains unmodified and is incorporated directly into the subproblem. It is well known and straightforward to verify that this procedure can be equivalently rewritten as
\[
    x_{k+1} = \operatorname{prox}_{\alpha h} \left( x_k - \alpha \nabla F(x_k) \right),
\]
with 
\[
    \operatorname{prox}_{\alpha h}(x): = \arg\min_{y} h(y) + \frac{1}{2\alpha} \norm{y - x}^2.
\]
For a comprehensive analysis of this method, readers can refer to the seminal works by Nesterov \cite{nesterov2013gradient} and Parikh \& Boyd \cite{parikh2014proximal}. Despite its widespread applicability and strong theoretical foundations, the proximal gradient method, like other first-order methods, their convergence speed is known to be slow. This limitation arises because first-order methods rely solely on gradient information, which may not fully capture the curvature of the objective function.
\medskip

Second-order methods such as Newton-type methods are particularly appealing because they achieve superlinear or even polynomial convergence rates while efficiently escaping saddle points \cite{agarwal2017finding}. The proximal Newton method extends classical Newton’s method to the composite setting \cite{lee2014proximal}. In this approach, the smooth function \( F \) is approximated by its second-order Taylor expansion, while the nonsmooth function \( h \) remains unchanged and is incorporated directly into the subproblem. This allows the method to effectively leverage curvature information from \( F \) while maintaining the structure of \( h \), leading to improved convergence behavior.  However, classical Newton’s method lacks global convergence guarantees, as its performance heavily depends on the initialization and the local behavior of the Hessian. In particular, when the Hessian is ill-conditioned or singular, Newton's updates can be unstable, leading to divergence or poor progress toward a solution. One approach to mitigate these issues is the regularized Newton method, which improves stability by adding a regularization term to the Hessian to ensure well-posed updates and better global behavior \cite{mishchenko2023regularized,doikov2024gradient}.  Another approach, cubic regularization, addresses these instability issues by introducing a cubic term to control the step size adaptively \cite{nesterov2006cubic}. Adaptive second-order methods
with cubic regularization and inexact steps were proposed in \cite{cartis2011adaptive,cartis2011adaptiveII}. Moreover, these approaches achieves global convergence rates that are faster than those of first-order methods.

\medskip
Higher-order methods extend the principles of first- and second-order methods by utilizing derivatives of the objective function beyond just the first (gradient) and second (Hessian) orders. These methods incorporate higher-order derivatives, such as third-order tensors or even higher, to build sophisticated higher-order Taylor models. The unpublished preprint \cite{baes2009estimate} stands as the first paper deriving theoretical results of the higher-order schemes for convex problems. However the extensive complexity associated with minimizing nonconvex multivariate polynomials has posed significant challenges, rendering this initial effort unsuccessful. Despite these obstacles, a ray of hope emerged through the groundbreaking research of Nesterov in \cite{nesterov2021implementable}. Specifically, Nesterov demonstrated that by appropriately regularizing the Taylor approximation, the auxiliary subproblem remains convex and can be solved efficiently, thereby offering a promising avenue for tackling convex unconstrained smooth problems. As researchers delve deeper into the intricacies of optimization methods, particularly within the nonconvex setting \cite{birgin2017worst,cartis2019universal}, a notable focus has been placed on analyzing the complexity of high-order approaches. These approaches aim to generate solutions with small gradient norms, a crucial aspect for navigating nonconvex optimization landscapes effectively. In addressing this challenge, it is essential for such methods to maintain a satisfactory adherence to first-order optimality conditions and ensure local reductions in the objective function.  Convergence guarantees, particularly in terms of the norm of the gradient, have been established to be of order $\mathcal{O}\left(k^{-\frac{p}{p+1}} \right)$ \cite{birgin2017worst,cartis2019universal}. High-order inexact tensor methods were considered in \cite{jiang2020unified,agafonov2024inexact,martinez2017high}. Increasing the order of the oracle, denoted as \( p \), provides certain advantages. Despite the complexity involved, ongoing research efforts continues to explore the performance of high-order optimization methods in nonconvex settings, aiming for more robust and efficient techniques.

\medskip
The aforementioned methods are generally monotone, meaning that they either require or implicitly ensure a monotonic decrease in the objective function. To maintain this monotonicity, it is typically necessary to assume that the smooth part of the objective function is globally Lipschitz continuous over the entire space or the prior assumption that the iterates generated by these methods remain bounded. Nonmonotone first-order methods have been developed, which relax these assumptions and allow for more flexible descent mechanisms. Unlike traditional monotone schemes which require or enforce a strict decrease in the objective function along the iterates, nonmonotone methods allow occasional increases in function value while still guaranteeing overall convergence (see, e.g., \cite{zhang2004nonmonotone,kanzow2022convergence,de2023proximal,kanzow2024convergence,qian2024convergence}). This flexibility can be particularly beneficial for composite optimization problems, where strict descent conditions might delay progress or slow convergence. For example, \cite{kanzow2024convergence} considered a nonmonotone proximal gradient method for minimizing problem \eqref{eq:problem}, assuming that \( \nabla F \) is locally Lipschitz, and \( h \) is bounded below by an affine function. Under these mild assumptions, \cite{kanzow2024convergence} derives global convergence rates that essentially have the same rate of convergence properties as its monotone counterparts. While existing works on higher-order methods typically assume global Lipschitz continuity of the higher-order derivatives, and their convergence analysis relies on establishing a strict descent in the objective function, in this paper, we propose a nonmonotone higher-order method that does not depend on these conditions. Instead, our approach relaxes the standard smoothness assumptions and allows for a more flexible descent mechanism, enabling improved performance in settings where global Lipschitz properties do not hold or are difficult to verify.

%One key assumptions in higher-order method is that the $p$ derivative is assumed to be \textit{globally} Lipschitz continuous. This condition poses significant limitations. \yass{Here they may complain the fact that you dont want to use third-order and instead you want second order. So why even mentioning about higher-order in the begging !!} For instant, our motivation comes from the following one rank optimization problem \cite{chi2019nonconvex}
%\begin{align*}
 %   \min_{x} \frac{1}{2}\norm{xx^T - M}^2.
%\end{align*}
%The objective function is a polynomial of order \( 4 \), meaning that only the third derivative is Lipschitz continuous. As a result, only third-order methods can be applied with a global convergence guarantee. While research is ongoing to efficiently solve this nonconvex third-order subproblem, one may wonder whether lower-order methods, such as second-order methods \cite{nesterov2006cubic}\yass{Check papers of Coralia if they dont consider local Lipschitz continuity}, could be used to solve this problem while still ensuring global convergence.\\
%\sloppy
\medskip
\textbf{Contributions.}
We propose NHOTA, a \textbf{N}onmonotone \textbf{H}igher-\textbf{O}rder \textbf{T}aylor \textbf{A}pproximation methods for solving composite optimization problems \eqref{eq:problem}. Our approach leverages higher-order Taylor approximations combined with nonmonotone techniques, allowing us to establish global convergence rates without requiring global Lipschitz continuity and/or boundedness of the iterates. Our contributions can be summarized as follows:
\begin{enumerate}
    \item We introduce NHOTA, a higher-order tensor  method that incorporates nonmonotonicity technique to solve composite problems \eqref{eq:problem}. 
    \item We derive global convergence guarantees for NHOTA under mild assumptions for nonconvex composite problems. Specifically, we show that the iterates generated by NHOTA converge to a stationary points and the convergence rate is of order $\mathcal{O}\left(k^{-\frac{p}{p+1}}\right)$, where $k$ is the iteration counter. Additionally, when the objective function satisfies a Kurdyka-Łojasiewicz (KL) property, we derive linear/sublinear convergence rates in function values, depending on the parameter of the KL condition.  
    \item For convex composite problems, we derive a global sublinear convergence rate in function values, and the convergence rate is of order $\mathcal{O}\left(k^{-p}\right)$. 
    \item  Our results are obtained without requiring global Lipschitz continuity assumptions, boundedness of the iterates, or the traditional monotonicity condition in the objective along the iterates, thereby broadening the applicability of NHOTA.
\end{enumerate}
\textbf{Contents.} The remaining parts of this paper are organized as follows. In Section \ref{sec:2}, we discuss the basic properties required for our analysis. In Section \ref{sec:3}, we introduce our method, NHOTA. Section \ref{sec:4} presents global/local convergence guarantees for NHOTA for nonconvex composite problems, followed by the analysis of the convergence rate for convex composite problems in Section \ref{sec:5}. In Section \ref{sec:6}, we present numerical illustrations of NHOTA on nonconvex phase retrieval problems. Finally, Section \ref{sec:7} concludes the paper and suggests directions for future work.

\medskip

%%%%%%%%%%%%%%%%%%%%%%%%%%%%%%%%%%%%%%%%%%%%%%%%
\section{Notations and preliminaries }\label{sec:2}
 We denote a finite-dimensional real vector space with $\mathbb{E}$ and $\mathbb{E}^{*}$ its dual space composed of linear functions on $\mathbb{E}$. For any linear function $s\in\mathbb{E}^{*}$, the value of $s$ at point $x\in\mathbb{E}$ is denoted by $\langle s,x\rangle$.  Using a self-adjoint positive-definite operator $B:\mathbb{E}\rightarrow \mathbb{E}^{*}$, we endow these spaces with conjugate Euclidean~norms:
\begin{align*}
    \norm{x}=\langle Bx,x\rangle^{\frac{1}{2}},\quad x\in \mathbb{E},\qquad \norm{g}_{*}=\langle g,B^{-1}g\rangle^{\frac{1}{2}},\quad g\in \mathbb{E}^{*}.
\end{align*}   
For a twice differentiable function $\phi$ on a convex and open domain $\text{dom}\;\phi \subseteq \mathbb{E}$, we denote by $\nabla \phi(x)$ and $\nabla^{2} \phi(x)$ its gradient and hessian evaluated at  $x\in \text{dom}\;\phi$, respectively. Then, $\nabla \phi(x)\in\mathbb{E}^{*}$ and $\nabla^{2}\phi(x)h\in\mathbb{E}^{*}$ for all $x\in\text{dom}\;\phi$, $h\in\mathbb{E}$.  Throughout the paper, we consider $p,q$ positive integers.  In what follows, we often work with directional derivatives of function $\phi$ at $x$ along directions $h_{i}\in \mathbb{E}$ of order $p$, $D^{p} \phi (x)[h_{1},\cdots,h_{p}]$, with $i=1:p$.  If all the directions $h_{1},\cdots,h_{p}$ are the same, we use the notation
$ D^{p} \phi(x)[h]$, for $h\in\mathbb{E}.$
If $\phi$ is $p$ times differentiable, then $D^{p} \phi (x)$ is a symmetric $p$-linear form, and its norm is defined as:
\[
    \norm{D^{p} \phi (x)}= \max\limits_{h\in\mathbb{E}} \left\lbrace D^{p} \phi (x)[h]^{p} :\norm{h}\leq 1 \right\rbrace.  
\] 
\noindent Further, the Taylor approximation of order $p$ of the function $\phi$ at $x\in \text{dom}\;\phi$ is denoted with:
\[
    T_p^{\phi}(y; x)= \phi(x) + \sum_{i=1}^{p} \frac{1}{i !} D^{i} \phi(x)[y-x]^{i} \quad \forall y \in \mathbb{E}.
\]
 Let $\phi: \mathbb{E}\to \bar{\mathbb{R}}$ be a $p$ differentiable function on the open domain $\text{dom}\;\phi$. Then, the $p$ derivative is locally Lipschitz continuous if for any compact set $X\subset \text{dom}\;\phi$, there exist a constant $L_{p,X}^{\phi} > 0$ such that the following relation holds:
\begin{equation} \label{eq:001}
    \| D^p \phi(x) - D^p \phi(y) \| \leq L_{p,X}^{\phi} \| x-y \| \quad  \forall x,y \in X.
\end{equation}
\noindent It is known that if \eqref{eq:001} holds,  then the residual between the function  and its Taylor approximation can be bounded \cite{nesterov2021implementable}:
    \begin{equation}\label{eq:TayAppBound}
    \vert \phi(y) - T_p^{\phi}(y;x) \vert \leq  \frac{L_{p,X}^{\phi}}{(p+1)!} \norm{y-x}^{p+1} \quad  \forall x,y \in X.
    \end{equation}
  
\noindent  If $p \geq 2$, we also have the following inequalities valid for all $ x,y \in X$:
    \begin{align} \label{eq:TayAppG1}
    &\| \nabla \phi(y) - \nabla T_p^{\phi}(y;x) \|_* \leq \frac{L_{p,X}^{\phi}}{p!} \|y-x \|^p, \\
    \label{eq:TayAppG2}
    &\|\nabla^2 \phi(y) - \nabla^2 T_p^{\phi}(y;x) \| \leq \frac{L_{p,X}^{\phi}}{(p-1)!} \| y-x\|^{p-1}.
    \end{align}
For the Hessian, the norm defined in \eqref{eq:TayAppG2} corresponds to the spectral norm of self-adjoint linear operator (maximal module of all eigenvalues computed w.r.t. $B$). %We say that $\phi$ is uniformly convex of degree $q$ with constant $\sigma$ if:
%\begin{align*}
%    \phi(y) \geq \phi(x) +\langle \nabla \phi(x), y-x \rangle + \frac{\sigma}{q}\|y-x\|^q,\;\forall x,y\in\text{dom}\;\phi.
%\end{align*}
For a convex function $h:\mathbb{E} \to \bar{\mathbb{R}}$, we denote by $\partial h(x)$ its subdifferential at $x$ that is defined by $\partial h(x) : =\{\lambda\in \mathbb{E}^*  : h(y) \geq h(x) + \langle \lambda , y-x\rangle\;\forall y\in \text{dom}\;h\}$. Denote $S_f(x) := \text{dist}(0,\partial f(x))$. Let us also recall the definition of a function satisfying the \textit{Kurdyka-Lojasiewicz (KL)} property (see \cite{bolte2007clarke} for more details). 

\begin{definition}
    \label{def:kl}
    \noindent A proper lower semicontinuous  function $f: \mathbb{E}\rightarrow (-\infty , +\infty]$ satisfies  \textit{Kurdyka-Lojasiewicz (KL)} property on the compact set $\Omega \subseteq \text{dom} \; f$ on which $f$ takes a constant value $f_*$ if there exist $\delta, \epsilon >0$ such that   one has:
    \begin{equation}\label{eq:kl0}
        \kappa' (f(x) - f_*) \cdot  S_{f}(x)  \geq 1  \quad   \forall x\!:  \text{dist}(x, \Omega) \leq \delta, \;  f_* < f(x) < f_* + \epsilon,  
    \end{equation}
    where $\kappa: [0,\epsilon] \mapsto \mathbb{R}$ is  concave differentiable function satisfying $\kappa(0) = 0$ and $\kappa'>0$.
\end{definition}    
This definition is satisfied by a large class of functions, for example, functions that are semialgebric (e.g., real polynomial functions), vector or matrix (semi)norms (e.g., $\|\cdot\|_p$ with $p \geq 0$ rational number), see \cite{bolte2007clarke} for a comprehensive list. For example, if $f$ is semi-algebraic function, then we have $\kappa (t) = \sigma_q^{\frac{1}{q}} \frac{q}{q-1} t^{\frac{q-1}{q}}$, with $q >1$ and $\sigma_q>0$ \cite{bolte2007lojasiewicz}. Then, the KL property establishes the following local geometry of the nonconvex function $f$ around a compact set~$\Omega$:
\begin{equation}\label{eq:kl}
f(x) - f_*  \leq \sigma_q  S_{f}(x)^q \quad   \forall x\!: \;  \text{dist}(x, \Omega) \leq \delta, \; f_* < f(x) < f_* + \epsilon. 
\end{equation}

%%%%%%%%%%%%%%%%%%%%%%%%%%%%%%%%%%%%%%%%%%%%%%%%
\section{Nonmonotone higher-order Tensor method}\label{sec:3}
In this section, we introduce our method along with the necessary assumptions for its analysis. To ensure clarity and simplicity in the presentation, we define the following notation, where \( M > 0 \) represents a constant that will be used throughout the discussion
\begin{align}\label{eq:mod_p}
s_{M}(y;x) \stackrel{\text{def}}{=} T_p^{F}(y;x) + \frac{M}{(p+1)!} \|y - x \|^{p+1} + h(x).
\end{align} 
For problem \eqref{eq:problem}, we consider the following assumptions
\begin{assumption}
    \label{ass:1}
    We have the following assumptions on the problem \eqref{eq:problem}
    \begin{enumerate}
    \item The function \( f \) is bounded from below over its domain.
    \item The function \( h \) is a proper, lower semicontinuous, and convex function that is bounded from below by an affine function.
    \item The function \( F \) is \( p \)-times differentiable, and \( D^p F \) is locally Lipschitz continuous.
    \end{enumerate}
\end{assumption}
%%%%%%%%%%%%%%%%%%%%%%%%%%%%%%%%%%%%%%%%%%%%%
Note that Assumption \ref{ass:1} is fundamental in analyzing nonmonotone schemes \cite{de2023proximal,kanzow2022convergence,kanzow2024convergence}. Indeed, the first condition is introduced to ensure the well-posedness of the problem \eqref{eq:problem}, while the local Lipschitz condition is equivalent to \( D^p F \) being Lipschitz continuous on any compact set. Moreover, this local Lipschitz property is a significantly weaker condition than the usual global Lipschitz assumption. For instance, the exponential function, the natural logarithm, and all polynomials of degree higher than \( p \) are locally Lipschitz but not globally Lipschitz on their respective domains. In higher-order methods, it is usually shown that the objective function is strictly decreasing along the iterates, or a backtracking technique is employed to ensure this decrease. A key observation is that the next iterate, \( x_{k+1} \), is required to satisfy the following strict descent inequality:
\[
    f(x_{k+1})\leq f(x_k) - \frac{M}{(p+1)!}\|x_{k+1} - x_k\|^{p+1},
\]
for a given \( M > 0 \). This is generally guaranteed under some global Lipschitz continuity of the \( p \)-th derivative of the smooth part of the objective, or by ensuring that the iterates remain bounded (e.g., \cite{birgin2017worst,nabou2024efficiency}). Building on the works of \cite{kanzow2024convergence,qian2024convergence,zhang2004nonmonotone}, our method is designed to enforce the following condition, which plays a crucial role in ensuring convergence
\[
    f(x_{k+1})\leq \mathcal{R}(x_k) - \frac{M}{(p+1)!}\|x_{k+1} - x_k\|^{p+1},
\]
%\yass{check the Nonmonotone proofs}
where the reference function \( \mathcal{R}(x_k) \) depends on past iterates. This formulation leverages historical information to provide a more flexible and potentially more effective descent mechanism compared to traditional approaches. Following \cite{kanzow2024convergence,zhang2004nonmonotone}, \( R(x_k) \)
is computed as a convex combination of the previous reference value $\mathcal{R}_{k-1}$ and the new function
value $f(x_{k})$. Our method is defined in Algorithm \ref{alg:NOHTA}.
%%%%%%%%%%%%%%%%%%%%%%%%%%%%%%%%%%%%%%%%%%%%%%%%
\begin{algorithm}
    \caption{ NHOTA}
    \label{alg:NOHTA}
    \begin{algorithmic}[1]
    \Require
    $x_0$, $M_0, \Tilde{M} > 0$ and $\mathcal{R}_0 = f(x_0)$.
    \State Set $k = 0$.
    \While{suitable criterion is not
    satisfied}
    \Repeat
    \label{line:line_search_nonm}
        \State For $i=0,2,\cdots$ compute $x^i_{k+1}$ satisfying 
            \begin{equation}\label{eq:solv_sub}
                \begin{split}
                &s_{M_k}(x^i_{k+1};x_k) + h(x^i_{k+1}) \leq f(x_k),\\
                &\|\nabla s_{M_k}(x^i_{k+1};x_k) + p_{k+1} \|\leq \theta \|x^i_{k+1} - x_k\|^p, 
                \end{split} \Comment{Inexact solution}
            \end{equation}
        for some $p_{k+1} \in \partial h(x^i_{k+1})$, with $M_k =\ 2^{i} M_k$, %until the following condition holds:
    \Until{
          $  f(x^i_{k+1}) \leq  \mathcal{R}_{k} - \frac{\Tilde{M}}{(p+1)!}\|x^i_{k+1} - x_k\|^{p+1}.$
    }\label{line:dec_nom}
    \State Denote $i_k$ by the termination criterion and set $x_{k+1} = x^{i_k}_{k+1}$ and $M_k = \max(M_k/2 , M_0)$.
    \State Choose $u_{k+1}\in(u_{\min},1]$ and set $\mathcal{R}_{k+1} = (1-u_{k+1})\mathcal{R}_k + u_{k+1} f(x_{k+1})$.
    \State Set $k=k+1$.
    \EndWhile
    \end{algorithmic}
\end{algorithm}
It is well known that if $F$ is convex and \( M \geq p L_p \), where \( L_p \) is the Lipschitz constant of the \( p \)-th derivative of \( F \), then the model defined in \eqref{eq:mod_p} is convex (see \cite{nesterov2021implementable}). However, in the absence of \( L_p \), there is no guarantee that the subproblem \eqref{eq:mod_p} remains convex for a given \( M >0 \), and thus the model \eqref{eq:mod_p} is generally nonconvex. Therefore, each iteration of NHOTA involves the approximate minimization of the model \( s_{M_k}(\cdot; x_k) \), as defined in \eqref{eq:mod_p}, satisfying conditions \eqref{eq:solv_sub}. It is important to note that this condition is relatively mild, as it only requires a decrease in the regularized \( p \)-th order model and the identification of an approximate first-order stationary point \cite{birgin2017worst,grapiglia2020tensor,martinez2017high}. No global optimization of this potentially nonconvex model is required.

%%%%%%%%%%%%%%s
\begin{comment}
\begin{center}
\noindent\fbox{%
\parbox{11.5cm}{%
\textbf{NHOTA method}
\begin{enumerate}
\item Given $x_0$, $i=1$, $M_0 \geq 0$ and $\mathcal{R}_0 = f(x_0)$. Set $k = 0$.\\
\item \textbf{While} suitable criterion is not satisfied do:
\item For $i=1,2,\cdots$ compute $x^i_{k+1}$ solution satisfying:
\begin{align*}
    &s_{M_k}(x^i_{k+1};x_k) + h(x^i_{k+1}) \leq f(x_k),\\
    & \|\nabla s_{M_k}(x^i_{k+1};x_k) + p^i_{k+1} \|\leq \theta \|x^i_{k+1} - x_k\|^p,
\end{align*}
for some $p^i_{k+1} \in \partial h(x^i_{k+1})$, with $M_k =\ 2^{i} M_k$ until the following condition holds:
\begin{align*}
  f(x^i_{k+1}) \leq  \mathcal{R}_{k} - \frac{M_k}{(p+1)!}\|x^i_{k+1} - x_k\|^{p+1}.  
\end{align*} 
\item Denote $i_k$ by the termination criterion and set $x_{k+1} = x^i_{k+1}$ and $M_k = \max(M_k/2 , M_0)$.
\item Choose $u_{k+1}\in(u_{\min},1]$ and set $\mathcal{R}_{k+1} = (1-u_{k+1})\mathcal{R}_k + u_{k+1}f(x_{k+1})$.
\item Set $k=k+1$.
\end{enumerate}
}
}
\end{center}
\end{comment}
%%%%%%%%%%%%%%%%%%%%%%%%%%%%%%%%%%%%%%%%%%%%%%%%
\section{Nonconvex convergence analysis}\label{sec:4}
In this section, we derive the global convergence results for the proposed method, establishing the conditions under which the method converges to a stationary point. The analysis will be based on the assumptions introduced earlier and will highlight the key factors contributing to the global convergence of the algorithm. 
\begin{lemma}
\label{lem:1}
Let Assumption \ref{ass:1} be satisfied and let $(x_k)_{k\geq 0}$ be generated by NHOTA \ref{alg:NOHTA}. Then, the following statements hold for all $k\geq 0$:
\begin{enumerate}
    \item $\mathcal{R}_k \geq f(x_k)$, and step 4 in algorithm \ref{alg:NOHTA} is well defined.
    \item The sequence $(\mathcal{R}_k )_{k\geq 0}$ is monotonically decreasing and satisfying
    \begin{equation}\label{eq:des_rk}
      \mathcal{R}_{k+1} \leq \mathcal{R}_k - \frac{u_{\min} \Tilde{M}}{(p+1)!}\norm{x_{k+1} - x_k}^{p+1}.  
    \end{equation}
\end{enumerate}
\end{lemma}
%%%%%%%%%%%%%%%
%%%%%%%%%%%%%%%
\begin{proof}
    %The proof follows in a similar reasoning as in \cite{kanzow2022convergence} adopted to our settings. 
    Let's prove (1). Suppose that the inner loop in step 4 does not terminate after a finite number of steps in iteration $k$. Recall that $x_{k+1}$ satisfies
    \begin{align*}
        T_{p}^F(x^i_{k+1};x_k) + \frac{M_k}{(p+1)!}\|x^i_{k+1} - x_k\|^{p+1} + h(x^i_{k+1}) \leq f(x_{k}).
    \end{align*}
    %This implies that:
    %\begin{align*}
    %    \sum_{j=1}^{p}\frac{1}{i!}D^j F(x_k)[x^i_{k+1} - x_k]^j + \frac{M_k}{(p+1)!}\|x^i_{k+1} - x_k\|^{p+1} + h(x^i_{k+1}) \leq h(x_{k}).
    %\end{align*}    
    Since \( M_k := 2^i M_k \to \infty \) as \( i \to \infty \) and given that \( h \) is bounded below by an affine function, we obtain that \( \|x^i_{k+1} - x_k\| \to 0 \) as \( i \to \infty \). Consequently, for any constant \( \delta > 0 \), there exists \( I > 0 \) such that for all \( i \geq I \), we have \( \|x^i_{k+1} - x_k\| \leq \delta \). Using the local Lipschitz property, we then deduce the existence of \( L_{I,\delta} > 0 \) satisfying for all $i\geq I$
    \[
        F(x_{k+1}^i) \leq T_{p}^F(x^i_{k+1};x_k) + \frac{L_{I,\delta}}{(p+1)!}\|x^i_{k+1} - x_k\|^{p+1}.
    \]
    Thus, summing up the last two inequality we obtain 
    \[
       f(x_{k+1}^i) \leq  f(x_{k}) - \frac{M_{k} - L_{I,\delta}}{(p+1)!}\|x^i_{k+1} - x_k\|^{p+1}.
    \]
    Further, there exist $\bar{I}$ such that $\forall i\geq \max(\bar{I},I)$ we have $M_k \geq \Tilde{M} + L_{I,\delta}$. This implies
    \begin{align}\label{eq:desc_obj}
        f(x_{k+1}^i) \leq  f(x_{k}) - \frac{\Tilde{M}}{(p+1)!}\|x^i_{k+1} - x_k\|^{p+1}.        
    \end{align}
    Next, it remains to prove $f(x_{k}) \leq \mathcal{R}_k$ for all $k\geq 0$. Indeed, by recurrence, for $k=0$, it follows immediately since $\mathcal{R}_0 = f(x_0)$. Assume that we have $f(x_{k}) \leq \mathcal{R}_k$. Then, combining \eqref{eq:desc_obj} with the update in Algorithm \ref{alg:NOHTA} (line 5) we get
    \begin{equation}\label{eq:desc_rk}
    \begin{split}
         f(x^i_{k+1}) &\leq f(x_{k}) - \frac{\Tilde{M}}{(p+1)!}\|x^i_{k+1} - x_k\|^{p+1}\\
                     &\leq \mathcal{R}_k - \frac{\Tilde{M}}{(p+1)!}\|x^i_{k+1} - x_k\|^{p+1}  
    \end{split}
    \end{equation}
    Therefore, combining this inequality with the updates of $\mathcal{R}_{k}$, we get:
    \begin{align*}
     \mathcal{R}_{k+1}  &= (1-u_{k+1})\mathcal{R}_k + u_{k+1}f(x^i_{k+1})\\
     &\geq   (1-u_{k+1})f(x^i_{k+1}) + u_{k+1}f(x^i_{k+1})  = f(x^i_{k+1}).
    \end{align*}
    This proves our first claim. Further, combining \eqref{eq:desc_rk} with the update of $\mathcal{R}_{k+1}$ yields
    \begin{equation*}
    \begin{split}
        &u_{k+1}f(x^i_{k+1}) \leq u_{k+1}\mathcal{R}_k -  \frac{u_{k+1}\Tilde{M}}{(p+1)!}\|x^i_{k+1} - x_k\|^{p+1}\\
        &\implies \mathcal{R}_{k+1}: = (1-u_{k+1})\mathcal{R}_k + u_{k+1}f(x^i_{k+1}) \leq \mathcal{R}_k -  \frac{u_{k+1}\Tilde{M}}{(p+1)!}\|x^i_{k+1} - x_k\|^{p+1}.
    \end{split}
    \end{equation*}
    Hence
    \begin{equation*}
      \mathcal{R}_{k+1} \leq \mathcal{R}_k - \frac{u_{\min} \Tilde{M}}{(p+1)!}\norm{x_{k+1} - x_k}^{p+1}.  
    \end{equation*}
    This proves our second assertion. \hfill\qed
\end{proof}
\begin{remark}
    Lemma \ref{lem:1} demonstrates that the sequence \( (\mathcal{R}_k)_{k \geq 0} \) is monotonically decreasing. Next, we will derive global convergence rate based on this descent, without requiring strict descent in the objective. Define the following level sets \( \mathcal{L}_f(x_0) = \{x : f(x) \leq f(x_0)\} \). Consequently, we have \( x_k \in \mathcal{L}_f(x_0) \) for all \( k \geq 0 \). Indeed, since \( \mathcal{R}_k \) is decreasing, we obtain
    \begin{equation*}
        f(x_{k}) \leq \mathcal{R}_{k}
        \leq \mathcal{R}_{k-1}\leq \cdots \leq \mathcal{R}_0: = f(x_0) \implies x_k \in \mathcal{L}_f(x_0). 
    \end{equation*}
\end{remark}
The following theorem derives global convergence of NHOTA under Assumption \ref{ass:1} and that the level sets $\mathcal{L}_{f}(x_0)$ is bounded, i.e., there exists $D>0$ such that $\text{diam}(\mathcal{L}_{f}(x_0))\leq D$.

\begin{theorem}\label{th:2}
Let Assumption \ref{ass:1} hold, and additionally assume that $\mathcal{L}_{f}(x_0)$ is bounded. Let $(x_k)_{k\geq 0}$ be generated by NOHTA algorithm \ref{alg:NOHTA}. Then, we have:
\begin{equation}
        \min_{i=0:k} \emph{dist}(0,\partial f(x_{i})) 
        \leq
        \left(\frac{L + p!\theta + M_{\max}}{p!}\right) \left(\frac{(p+1)!(f(x_0) - f^*)}{u_{\min}\Tilde{M} \;k}\right)^{\frac{p}{p+1}}.
\end{equation}
\end{theorem}
%%%%
\begin{proof}
    Since $\mathcal{L}_{f}(x_0)$ is bounded, then the sequence $(x_k)_{k\geq 0}$ generated by NHOTA is bounded. Using Assumption \ref{ass:1} and \eqref{eq:TayAppG1}, there exists $L>0$ such that
    \begin{align*}
        \norm{\nabla F(x_{k+1}) - \nabla T_p^F(x_{k+1};x_k)} \leq \frac{L}{p!}\|x_{k+1} - x_k\|^{p}.
    \end{align*}
    Then, we get for all $p_{k+1} \in\partial h(x_{k+1})$
    \begin{equation*}
        \begin{split}\label{eq:bounded_sub_iterates}
            \norm{\nabla F(x_{k+1}) + &p_{k+1}}
            =
            \norm{\nabla F(x_{k+1}) -\nabla T_p^F(x_{k+1};x_k) + \nabla T_p^F(x_{k+1};x_k)+ p_{k+1}}\\
            &
            \leq 
            \norm{\nabla F(x_{k+1}) -\nabla T_p^F(x_{k+1};x_k)} + \norm{\nabla s_{M_k}(x_{k+1};x_k)+ p_{k+1}} \\
            \MoveEqLeft[-1]
            +  \frac{M_k}{p!}\norm{x_{k+1} - x_k}^p\\
            &
            \leq
            \frac{L}{p!}\norm{x_{k+1} - x_k}^{p} + \theta\norm{x_{k+1} - x_k}^p + \frac{M_{\max}}{p!}\norm{x_{k+1} - x_k}^p\\
            &
            = \frac{L + p!\theta + M_{\max}}{p!}\norm{x_{k+1} - x_k}^p.
        \end{split}
    \end{equation*}
    Further, combining this inequality with the descent \eqref{eq:des_rk}, we obtain
    \begin{equation}\label{eq:bound:dist}
        \begin{split}
        \left(\text{dist}(0,\partial f(x_{k+1})) \right)^{\frac{p+1}{p}}
        &\leq \left(\frac{L + p!\theta + M_{\max}}{p!}\right)^{\frac{p+1}{p}} \norm{x_{k+1} - x_k}^{p+1}\\
        & \left(\frac{L + p!\theta + M_{\max}}{p!}\right)^{\frac{p+1}{p}} \frac{(p+1)!}{u_{\min}\Tilde{M}} (\mathcal{R}_k - \mathcal{R}_{k+1}).
        \end{split}
    \end{equation}
    Summing up this inequality, we get
    \begin{equation*}
        \begin{split}
        \min_{i=0:k} k \left(\text{dist}(0,\partial f(x_{i})) \right)^{\frac{p+1}{p}}
        &
        \leq \sum_{i=0}^{k} \left(\text{dist}(0,\partial f(x_{i})) \right)^{\frac{p+1}{p}}\\
        &\left(\frac{L + p!\theta + M_{\max}}{p!}\right)^{\frac{p+1}{p}} \frac{(p+1)!}{u_{\min}\Tilde{M}} \sum_{i=0}^{k} (\mathcal{R}_i - \mathcal{R}_{i+1})\\
        & = \left(\frac{L + p!\theta + M_{\max}}{p!}\right)^{\frac{p+1}{p}} \frac{(p+1)!}{u_{\min}\Tilde{M}}(\mathcal{R}_0 - \mathcal{R}_{k+1})\\
        &
        \leq \left(\frac{L + p!\theta + M_{\max}}{p!}\right)^{\frac{p+1}{p}} \frac{(p+1)!}{u_{\min}\Tilde{M}}(f(x_0) - f^*),
        \end{split}
    \end{equation*}
    where the last inequality follows from $\mathcal{R}_k \geq f(x_k) \geq f^*$ for all $k\geq 0$. Hence
    \begin{equation*}
        \min_{i=0:k} \text{dist}(0,\partial f(x_{i})) 
        \leq
        \left(\frac{L + p!\theta + M_{\max}}{p!}\right) \left(\frac{(p+1)!(f(x_0) - f^*)}{u_{\min}\Tilde{M} k}\right)^{\frac{p}{p+1}}.
    \end{equation*}
    This proves our assertion. \hfill\qed
\end{proof}
\begin{remark}
Theorem \ref{th:2} proves that the iterates generated by NHOTA converge to a stationary point and the convergence rate is of order \( \mathcal{O}\left(k^{-\frac{p}{p+1}}\right) \), which is the standard convergence rate for higher-order algorithms for (unconstrained) nonconvex problems using higher-order derivatives \cite{birgin2017worst,cartis2019universal,nabou2024efficiency}. In our convergence analysis, we emphasize that we do not rely on any global Lipschitz continuity assumptions or a strict descent condition in the objective function. As a result, our approach is more general and applicable to a broader class of problems, making it more flexible compared to traditional methods that require such conditions.     
\end{remark}

%%%%%%%%%%%%%%%%%%%%%%%%%%%%%%%%%%%%%%%%%
\subsection{Improved convergence rates under KL}\label{sub_kl}

In this section, we derive convergence rates for the proposed method under the Kurdyka-Łojasiewicz (KL) property \eqref{def:kl}. To achieve this, we assume that the objective \( f \) is a continuous function, meaning that \( h \) is continuous. We establish improved convergence rates, proving linear or sublinear convergence in function values for the sequence \( (x_k)_{k \geq 0} \) generated by NHOTA. We denote the set of limit points of \( (x_k)_{k \geq 0} \) by \( \Omega(x_0) \).

\begin{align*}
\Omega (x_{0}) = &\lbrace \bar{x}\in \mathbb{R}^n: \exists (k_{t})_{t\geq0} \text{ $\nearrow$ },
\text{ such that } x_{k_{t}}\to \bar{x} \text{ as } t\to \infty \rbrace.
\end{align*} 

\noindent The next lemma derives some properties for $\Omega (x_{0})$.

\begin{lemma}\label{lem:3}
Let the assumptions of Theorem \ref{th:2} hold.  Additionally, assume that $f$ is continuous. Then, we have $\emptyset \neq \Omega(x_{0}) \subseteq\texttt{stat}\, f : = \{x:\; 0\in\partial f(x)\}$, $ \Omega(x_{0})$ is compact and connected set, and $f$ is constant on $ \Omega(x_{0})$, i.e., $f(\Omega(x_{0})) = f_*$.
\end{lemma}
%%%%%%%%%%%%%%%%%%%%%%%%%
\begin{proof}
Let us prove that $f(\Omega(x_{0}))$ is constant. From the descent \eqref{eq:des_rk} we have that $\left( \mathcal{R}_k\right)_{k\geq 0}$ is monotonically decreasing, and since $f$ is assumed to be bounded from below and that $\mathcal{R}_k \geq f(x_k)$, it converges. Let us say to $f_*>-\infty $, i.e., $\mathcal{R}_k\to f_*$ as $k\to \infty$. Further, we have $\mathcal{R}_{k+1} - \mathcal{R}_k = -u_{k+1}(\mathcal{R}_k - f(x_{k+1}))$, then
\[
    \mathcal{R}_{k} - \mathcal{R}_{k+1} = u_{k+1}(\mathcal{R}_k - f(x_{k+1})) \geq u_{\min}(\mathcal{R}_k - f(x_{k+1})) \geq 0.  
\]
Since we have $\mathcal{R}_{k} - \mathcal{R}_{k+1} \to 0$, then $f(x_{k}) \to f_*$ as $k\to\infty$. On the other hand, let $x_{*}$ be a limit point of the sequence $\left(x_{k}\right)_{k\geq0}$. This means that there exists a subsequence $\left( x_{k_{t}}\right)_{t\geq0}$ such that $x_{k_{t}}\to x_{*}$. Since $f$ is continuous, we get $f(x_{k_{t}})\to f(x_{*}) = f_*$ and hence, we have $f(\Omega(x_{0}))=f_*$. The closeness property of $\partial f$ implies that $S_{f}(x_{*})=0$, and thus $0\in \partial f(x_{*})$. This proves that $x_{*}$ is a stationary point of $f$ and thus $\Omega(x_{0})$ is nonempty. By observing that $\Omega(x_{0})$ can be viewed as an intersection of compact sets, $\Omega(x_{0})=\cap_{q\geq 0} \overline{\cup_{k\geq q}\lbrace x_{k}\rbrace}$ so it is also compact. The connectedness follows from \cite{bolte2014proximal}. This completes the proof.  
\end{proof}

%%%%%%%%%%%%%%%%%%%%%%%%%

%\noindent From Theorem \ref{th:2}, we have that the sequences $(x_k)_{k\geq 0}$ and $(y_k)_{k\geq 0}$ share the same limit point. Consequently,
\noindent Next, we derive improved convergence rates in function values for the sequence $(x_k)_{k\geq 0}$ generated by NHOTA.
\begin{theorem}\label{th:5}
    Let the assumptions of Lemma \ref{lem:3} hold. Additionally, assume that $f$ satisfy the KL property \eqref{eq:kl} on $\Omega(x_0)$. Then, the following convergence rates hold for $(x_{k})_{k\geq 0}$ generated by NHOTA for $k$ sufficiently large:
    \begin{enumerate}
    \item If $q\geq\frac{p+1}{p}$, then $f(x_{k})$ converges to ${f_*}$ linearly. 
    \item If $q < \frac{p+1}{p} $, then $f(x_{k})$ converges to ${f_*}$ at sublinear rate of order $\mathcal{O}\left(\frac{1}{k^{\frac{pq}{p+1-pq}}}\right)$.
    \end{enumerate}  
\end{theorem}
\begin{proof}
    From the definition of $\mathcal{R}_k$, we have
    \begin{equation*}
        \begin{split}
        \mathcal{R}_{k+1} - f^* &= (1-u_{k+1})\mathcal{R}_k + u_{k+1}f(x_{k+1}) - f^*\\
        & = (1-u_{k+1})(\mathcal{R}_k  - f^*)+ u_{k+1}(f(x_{k+1}) - f^*)\\
        &\leq (1-u_{k+1})(\mathcal{R}_k  - f^*)+ u_{k+1} \sigma_q S_f(x_{k+1})^q\\
        &\leq (1-u_{k+1})(\mathcal{R}_k  - f^*)+ u_{k+1} \sigma_q \frac{L + p!\theta M_{\max}}{p!}\norm{x_{k+1} - x_k}^{qp}\\
        & \leq (1-u_{k+1})(\mathcal{R}_k  - f^*) \\
         \MoveEqLeft[-1] + u_{k+1} \sigma_q \frac{L + p!\theta M_{\max}}{p!} \left(\frac{(p+1)!}{u_{\min} \Tilde{M}} \right)^{\frac{pq}{p+1}}\left(\mathcal{R}_k - \mathcal{R}_{k+1}\right)^{\frac{pq}{p+1}},
        \end{split}
    \end{equation*}
    where the first equality follows from the definition of $\mathcal{R}_{k+1}$, the first inequality follows from \eqref{eq:kl}, the second inequality follows from \eqref{eq:bound:dist}, the third inequality follows from the. descent \eqref{eq:des_rk}
    Then, we get
    \[
        \mathcal{R}_{k+1} - f^* \leq (1-u_{k+1})(\mathcal{R}_k  - f^*) + \kappa_q \left(\mathcal{R}_k - \mathcal{R}_{k+1}\right)^{\frac{pq}{p+1}},
    \]
    where $\kappa_q = \sigma_q \frac{L + p!\theta M_{\max}}{p!} \left(\frac{(p+1)!}{u_{\min} \Tilde{M}} \right)^{\frac{pq}{p+1}}$. 
    Let us  denote $\delta_{k} = \mathcal{R}_k - f_*$. Thus we obtain
    \[
        \delta_{k+1} \leq (1-u_{k+1})\delta_{k} +  \kappa_q \left( \delta_{k}-\delta_{k+1}\right)^{\frac{qp}{p+1}},
    \]
    rearranging the above inequality, it follows that
    \[
        u_{\min}\delta_{k+1} \leq u_{\min}\delta_{k} \leq u_{k+1}\delta_k \leq \left( \delta_{k} - \delta_{k+1}\right) + \kappa_q \left( \delta_{k}-\delta_{k+1}\right)^{\frac{qp}{p+1}}.
    \]
    Subsequently, we derive the following recurrence
    \begin{align*}
        \delta_{k+1} \leq \frac{1}{u_{\min}}\left( \delta_{k} - \delta_{k+1}\right) + \frac{\kappa_q}{u_{\min}} \left( \delta_{k}-\delta_{k+1}\right)^{\frac{qp}{p+1}}. 
    \end{align*}
    Using Lemma 6 in \cite{nabou2024efficiency} with $\theta = \frac{p+1}{pq}$ and that $f(x_k) - f_* \leq \mathcal{R}_k - f_*$, our assertions follow. \hfill\qed
\end{proof}
\begin{remark}
    In this section, we have derived improved convergence rates in terms of function values for  sequence $(x_k)_{k\geq 0}$ generated by NHOTA,  by leveraging higher-order information to solve  problem  \eqref{eq:problem}. As mentioned in our previous remark, our convergence analysis does not rely on global Lipschitz continuity assumptions or a strict descent condition in the objective function, making our approach more general and applicable to a broader class of nonconvex optimization problems. 
\end{remark}
%%%%%%%%%%%%%%%%%%%%%%%%%%%%%%%%%%%%%%%%%%%%%%%%%%%%%%%%%%%%%%%%%%%%%%%%%%%%%%%%%%
%%%%%%%%%%%%%%%%%%%%%%%%%%%%%%%%%%%%%%%%%%%%%%%%%%%%%%%%%%%%%%%%%%%%%%%%%%%%%%%%%%
%%%%%%%%%%%%%%%%%%%%%%%%%%%%%%%%%%%%%%%%%%%%%%%%%%%%%%%%%%%%%%%%%%%%%%%%%%%%%%%%%%

\section{Convex convergence analysis}\label{sec:5}
In this section, we assume that \( f \) is a convex function, and we define \( f^* = \min_{x} f(x) \) as the global minimum of \( f \). Convexity ensures that any local minimum is also a global minimum, making the problem well-posed. 
We aim to establish global convergence rate for the iterates generated by NHOTA \ref{alg:NOHTA} in the context of convex composite optimization problems. Specifically, we demonstrate that the iterates converge to the global minimum and establish a sublinear convergence rate function values. Recall that $\delta_k = \mathcal{R}_k - f^*$ and define the constants $C: =\frac{D(L + p!\theta + M_{\max})}{p!}\left(\frac{(p+1)!}{u_{\min}\Tilde{M}}\right)^{\frac{p}{p+1}} $, $\mu_0 =  \left(\frac{u_{\min}}{1 + C}\right)^{p+1} \delta_0$ . We now present the global convergence rate for convex composite problems.
\begin{theorem}\label{th:3}
Let Assumption \ref{ass:1} be satisfied and additionally assume that $\mathcal{L}_{f}(x_0)$ is bounded. Let $f$ be convex function and $(x_k)_{k\geq 0}$ be generated by NHOTA. Then, the sequence $(\delta_k - \delta_{k+1})$ converges to $0$. Moreover, if $(\delta_k - \delta_{k+1}) \geq 1$ then have the following convergence rate  
\begin{equation}
    f(x_k) - f^* \leq \left(\frac{\frac{1+C}{u_{\min}}}{ 1 + \frac{1+C}{u_{\min}}}\right)^{k}\delta_0.
\end{equation}
Otherwise, if $(\delta_k - \delta_{k+1}) < 1$ we have
\begin{equation}
    f(x_k) - f^*  \leq \frac{((1 + p)(1 + \mu_0^{\frac{1}{p}}))^p}{k^{p}}. 
\end{equation}
\end{theorem}
\begin{proof}
    Note that since $\mathcal{R}_k$ is decreasing and $\delta_k \geq 0$, then $\delta_k$ converges and thus  $(\delta_k - \delta_{k+1})\to 0$. Further,
    from the update of $\mathcal{R}_k$ (line 7 in NHOTA \ref{alg:NOHTA}) and the convexity of $f$, we get 
    \begin{equation}
    \begin{split}
        \mathcal{R}_{k+1} - f^* &= (1-u_{k+1})\mathcal{R}_k + u_{k+1}f(x_{k+1}) - f^*\\
       & = (1-u_{k+1})(\mathcal{R}_k  - f^*)+ u_{k+1}(f(x_{k+1}) - f^*)\\
       &\leq (1-u_{k+1})(\mathcal{R}_k  - f^*)+ u_{k+1}\left(\langle \nabla F(x_{k+1}) + p_{k+1} ,x_{k+1} - x^* \rangle \right)\\
        &\leq (1-u_{k+1})(\mathcal{R}_k  - f^*)+ u_{k+1} \left(\norm{\nabla F(x_{k+1}) + p_{k+1}}\norm{x_{k+1} - x^*} \right)\\
        &\leq (1-u_{k+1})(\mathcal{R}_k  - f^*)+ D u_{k+1}\norm{\nabla F(x_{k+1}) + p_{k+1}},
    \end{split}
    \end{equation}
    for $p_{k+1} \in\partial h(x_{k+1})$, the last inequality follows from the boundedness of diameter of $\mathcal{L}_f (x_0)$. Then combining the last inequality with \eqref{eq:bounded_sub_iterates}, we get
    \begin{equation*}
    \begin{split}
        \mathcal{R}_{k+1} & - f^*
        \leq (1-u_{k+1})(\mathcal{R}_k  - f^*) + \frac{D(L + p!\theta + M_{\max})}{p!}\norm{x_{k+1} - x_k}^p\\
        & 
        \leq (1-u_{k+1})(\mathcal{R}_k  - f^*) + \frac{D(L + p!\theta + M_{\max})}{p!}\left(\frac{(p+1)!}{u_{\min}\Tilde{M}} (\mathcal{R}_{k} - \mathcal{R}_{k+1}) \right)^{\frac{p}{p+1}}\\
        & = (1-u_{k+1})(\mathcal{R}_k  - f^*) + C(\mathcal{R}_{k} - \mathcal{R}_{k+1})^{\frac{p}{p+1}}.
    \end{split}
    \end{equation*}
    Denote $\delta_k: = \mathcal{R}_{k} - f^* \geq f(x_k) - f^* \geq 0$. As a result, we obtain
    \begin{equation*}
       \delta_{k+1} \leq (1-u_{k+1})\delta_k + C(\delta_k - \delta_{k+1})^{\frac{p}{p+1}},
    \end{equation*}
     implies
    \begin{equation}
       u_{k+1}\delta_{k} \leq (\delta_k -\delta_{k+1})  + C(\delta_k - \delta_{k+1})^{\frac{p}{p+1}}. 
    \end{equation}
    Furthermore, since $u_{\min} \leq u_{k+1}$ and $\mathcal{R}_{k+1} \leq \mathcal{R}_k$, then $0\leq \delta_{k+1} \leq \delta_k$ and thus we obtain the following recurrence
    \begin{equation}
       \delta_{k+1} \leq \frac{1}{u_{\min}}(\delta_k -\delta_{k+1})  + \frac{C}{u_{\min}}(\delta_k - \delta_{k+1})^{\frac{p}{p+1}}. 
    \end{equation}
    Given that $\delta_k$ is decreasing and bounded from below by $0$, then  $(\delta_k -\delta_{k+1}) \to 0$. Let as consider the following two cases. If $(\delta_k -\delta_{k+1}) \geq 1$, then we get
    \[
        \delta_{k+1} \leq \frac{1+C}{u_{\min}} (\delta_k -\delta_{k+1}) \implies     \delta_{k+1} \leq \frac{\frac{1+C}{u_{\min}}}{ 1 + \frac{1+C}{u_{\min}}} \delta_k.
    \]
    Consequently, it follows
    \[
        f(x_k) - f^* \leq \delta_k \leq \left(\frac{\frac{1+C}{u_{\min}}}{ 1 + \frac{1+C}{u_{\min}}}\right)^{k}\delta_0.
    \]
    On the other hand, if $(\delta_k -\delta_{k+1}) < 1$ then
    \[
        \delta_{k+1} \leq \frac{1+C}{u_{\min}} (\delta_k -\delta_{k+1})^{\frac{p}{p+1}} \implies 
        \left(\frac{u_{\min}}{1 + C}\right)^{\frac{p+1}{p}} \delta_k^{\frac{p}{p+1}} \leq  \delta_k -\delta_{k+1}.
    \]
    Thus, for $\mu_k =  \left(\frac{u_{\min}}{1 + C}\right)^{p+1} \delta_k$, the recursive inequality is as follows
    \[
        \mu_k^{\frac{p}{p+1}} \leq  \mu_k -\mu_{k+1}.
    \]
    Using lemma A.1 in \cite{nesterov2022inexact} with $\alpha = \frac{1}{p}$, we get
    \begin{equation*}
       f(x_k) - f^* \leq \delta_{k} \leq \frac{((1 + p)(1 + \mu_0^{\frac{1}{p}}))^p}{k^{p}}. 
    \end{equation*}
    This proves our claim. \hfill\qed
\end{proof}
\begin{remark}
In Theorem \ref{th:3}, we establish convergence rate guarantees based on the quantity \( \delta_k - \delta_{k+1} \). Specifically, if \( \delta_k - \delta_{k+1} \geq 1 \), we obtain a linear convergence rate in function value. Conversely, if \( \delta_k - \delta_{k+1} < 1 \), we prove a sublinear convergence rate of order \( \mathcal{O}(k^{-p}) \) for convex composite optimization problems, which aligns with the typical convergence rates of higher-order methods for convex (unconstrained) problems \cite{nesterov2021implementable,nabou2024efficiency}. As mentioned in our previous remark, our convergence analysis does not rely on global Lipschitz continuity assumptions or a strict descent condition in the objective function, making our approach more general and applicable to a broader class of convex optimization problems of the form \eqref{eq:problem}.  
\end{remark}

%%%%%%%%%%%%%%%%%%%%%%%%%%%%%%%%%%%%%%%%%%%%%%%%%%%%%%%%%%%%%%%%%%%%%%%%%%%%

\section{Numerical illustrations}\label{sec:6}
In this section, we evaluate the performance of the proposed method for nonconvex phase retrieval problems \cite{candes2015phase} with $l_1$ norm regularization. 
Our implementation was carried out on a MacBook M2 with 16GB of RAM, using the Julia programming language.   
\subsection{Nonconvex phase retrieval}
In our experiments we apply NHOTA for \( p = 2 \) to solve the following nonconvex phase retrieval problem with \(\ell_1\) regularization:  
\begin{equation}
    \min_{x \in \mathbb{R}^n} f(x) = \frac{1}{2m} \sum_{i=1}^{m} \left(y_i - (a_i^\top x)^2 \right)^2 + \lambda \|x\|_1,
\end{equation}  
where \( y_i \in \mathbb{R} \) represents the observed measurements, and \( a_i \in \mathbb{R}^n \) are the sensing vectors. This problem follows the structure of the composite problems \eqref{eq:problem} with \( F(x) = \frac{1}{2m} \sum_{i=1}^{m} \left(y_i - (a_i^\top x)^2 \right)^2 \) and \( h(x) = \lambda \|x\|_1 \). Since the function \( x \mapsto \left(y_i - (a_i^\top x)^2 \right)^2 \) is a polynomial of degree four, its Hessian is not globally Lipschitz continuous, which fits our setting as we do not require this property. The data is generated as follows: we set \( n = 100 \), \( m = 5000 \), and  
\( a_i, z \sim \mathcal{N}(0,0.5)\), \(
x_0 \sim \mathcal{N}(0,1)\)
where the samples are generated element-wise, with \( z \) denoting the true underlying object. The measurements are generated as  
\[
    y_i = (a_i^\top z)^2 + n_i,\; \text{for}\; i=1:m,
\]
where \( n_i \) represents random noise, and we set the regularization parameter to \( \lambda = 10^{-5} \). The proposed method is evaluated for different fixed values of \( u_k \in \{0.05, 0.25, 0.5, 0.75, 1\} \) across iterations, for all \( k \geq 0 \), as specified in line 7 of Algorithm \ref{alg:NOHTA}. All subproblems are solved using IPOPT \cite{wachter2006implementation}. The results, shown in Figure \ref{fig:convergence2}, indicate that all choices of \( u \) lead to convergence to a stationary point and a significant reduction in the objective function, depending on the noise level. Notably, when \( u = 1 \), the objective function decreases monotonically, whereas for \( u < 1 \), it does not necessarily follow a strictly decreasing trend. However, despite this lack of strict monotonicity, the final accuracy remains comparable to that of the monotone case, suggesting that allowing some flexibility in descent does not compromise solution quality. Regarding computational efficiency, it is difficult to determine an optimal choice for \( u \), but cases where \( u < 1 \) tend to perform better in terms of CPU time. This is likely because the algorithm does not enforce strict monotonicity in these cases, potentially allowing for larger step sizes and reducing the number of backtracking steps, leading to lower overall computational cost.
We run the proposed method for different fixed values of \( u_k := u \in \{0.05, 0.25, 0.5, 0.75, 1\} \) for all \( k \geq 0 \), as specified in line 7 of Algorithm \ref{alg:NOHTA}. The stopping criterion used is either \( f(x_k) \leq 10^{-3} \) or \( \|\nabla F(x_k)\| \leq 10^{-3} \)(we omit the term \( \|\nabla F(x_k) + \lambda g_k\| \approx \|\nabla F(x_k)\| \) for all \( g \in \partial h(x_k) \) due to \( \lambda = 10^{-5} \) and \( \|g_k\| \leq 10 \), as its contribution is negligible). All subproblems are solved using IPOPT \cite{wachter2006implementation}, and the results are presented in Figure \ref{fig:convergence2}.
Our first remark is that regardless of the chosen $u$, all cases successfully converge to a stationary point and achieve a significant reduction in the objective function/norm of the gradient depending of the noise level. A key observation is that when \( u < 1 \), the objective function does not strictly decrease along the iterates, in contrast to the case when \( u = 1 \), where it exhibits monotonic descent. However, despite this lack of strict monotonicity, the non-monotone cases (\( u < 1 \)) produce a first-order solution similar to that of the monotone setting (\( u = 1 \)). This suggests that allowing some flexibility in the descent process does not necessarily compromise solution quality. Finally, when analyzing performance with respect to computational time, it is difficult to determine a clear optimal choice for \( u \). Interestingly, some cases where \( u < 1 \) tend to perform better than \( u = 1 \) in terms of CPU efficiency. This can be attributed to the fact that when \( u < 1 \), the algorithm does not strictly enforce a monotonic decrease in the objective function. At the same time, it does not necessarily require more iterations than the monotone case, which may allow for larger step sizes in certain iterations. As a result, the backtracking procedure may take fewer steps in these cases, leading to a reduction in overall computation time. In conclusion, we believe that for certain applications, using smaller values of \( u \) (i.e., \( u < 1 \)) could potentially achieve better accuracy than the fully monotone case (\( u = 1 \)). However, we cannot provide a definitive answer at this stage. In future work, we plan to investigate this approach across different applications to gain a deeper understanding of its impact.

%%%%%%%%%%%%%%%%%%%%%%%%%%%%%%%%%%%%%%%%
\begin{comment}
\begin{figure}[h]
    \centering
    \begin{subfigure}[b]{0.49\textwidth}
        \centering
        \includegraphics[width=\textwidth]{plots/convergence_plots_u_values01.eps}
        \caption{level noise $n_i \sim \mathcal{N}(0,5)$}
        \label{fig:sub1}
    \end{subfigure}
    \hfill
    \begin{subfigure}[b]{0.49\textwidth}
        \centering
        \includegraphics[width=\textwidth]{plots/convergence_plots_u_values02-2.eps}
        \caption{level noise $n_i \sim 0.01\times\mathcal{N}(0,5)$}
        \label{fig:sub2}
    \end{subfigure}
    \caption{Behavior of NHOTA in terms of norm gradient and objective along iterations and CPU time.}
    \label{fig:convergence}
\end{figure}
\end{comment}
%%%%%%%%%%%%%%%%%%%%%%%%%%%%%%%%%%%%
\begin{figure}[ht]
    \centering
    \begin{subfigure}[b]{0.49\textwidth}  % 
        \centering
        \includegraphics[width=\textwidth, height=\textheight, keepaspectratio]{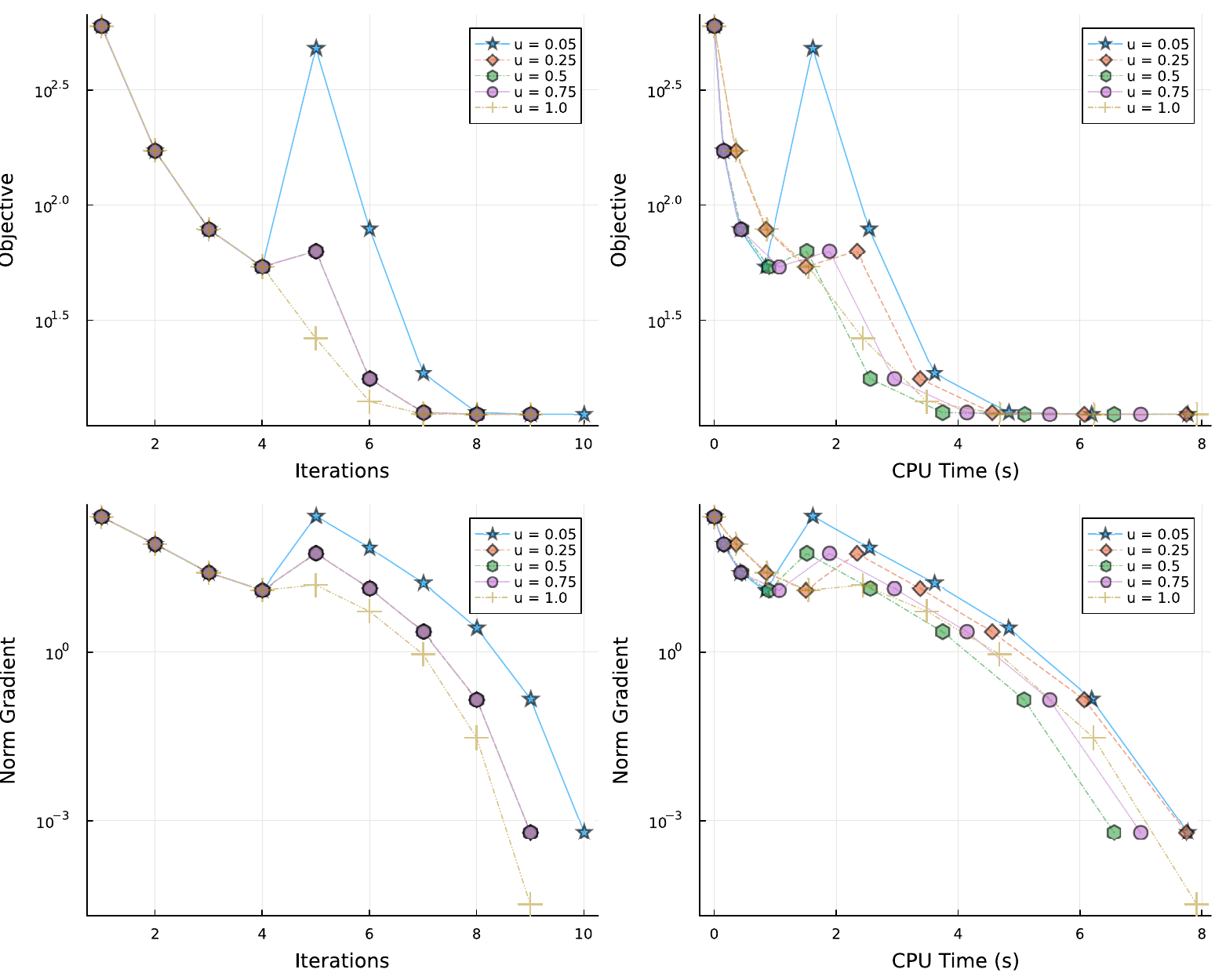}
        \caption{Level noise $n_i \sim \mathcal{N}(0,5)$}
        \label{fig:sub1}
    \end{subfigure}
    \hfill
    \begin{subfigure}[b]{0.49\textwidth}  % 
        \centering
        \includegraphics[width=\textwidth, height=\textheight, keepaspectratio]{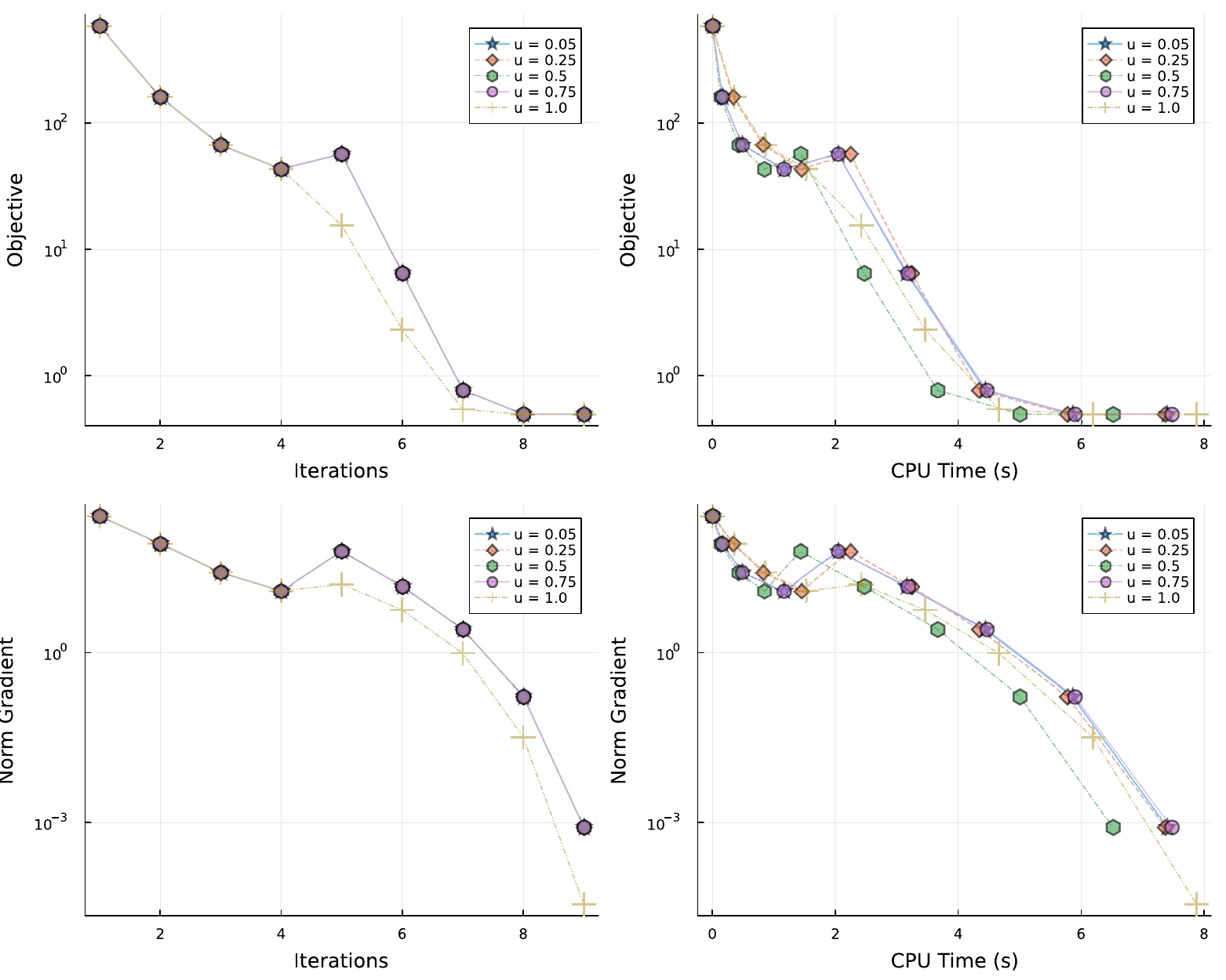}
        %{plots/convergence_plots_u_values02-3.eps}
        \caption{Level noise $n_i \sim \mathcal{N}(0,1)$}
        \label{fig:sub2}
    \end{subfigure}
    \caption{Performance of NHOTA in terms of the gradient norm \( \|\nabla F(x_k)\| \), objective value \(f(x_k)\), and CPU time across iterations. The term \( \|\nabla F(x_k) + \lambda g_k\| \approx \|\nabla F(x_k)\| \) for all \( g \in \partial h(x_k) \) is omitted due to \( \lambda = 10^{-5} \) and \( \|g_k\| \leq 10 \), as its contribution is negligible.}
    \label{fig:convergence2}
\end{figure}

%%%%%%%%%%
%\clearpage
%%%%%%%%%%%%
\section{Conclusion}\label{sec:7}
In this paper, we propose NHOTA, a Nonmonotone Higher-Order Taylor Approximation method for solving composite optimization problems. Our approach leverages higher-order tensor methods while avoiding traditional assumptions such as global Lipschitz continuity or strict descent conditions on the objective function. We derive global convergence guarantees for convex and nonxonvex composite problems, demonstrating that NHOTA achieves a convergence rate of order \( \mathcal{O}(k^{-\frac{p}{p+1}}) \) in the nonconvex setting,  with improved rates under the Kurdyka-Łojasiewicz (KL) property, and achieves sublinear convergence rate of order \( \mathcal{O}(k^{-p}) \) in the convex case. These results highlight that NHOTA maintains the efficiency of classical higher-order methods while offering flexibility and broader applicability.
Several directions remain open for future research. First, extending NHOTA to composition, smooth constrained, and nonlinear least squares optimization problems \cite{nabou2024efficiency,nabou2024moving,Yass2024} and analyzing its behavior would provide valuable insights. Furthermore, developing an accelerated variant of the proposed method for convex problems presents an interesting direction for future research \cite{nesterov2021implementable,nesterov2022inexact,grapiglia2020tensor}. Another promising direction is the efficient implementation of NHOTA on modern hardware to accelerate large-scale optimization tasks, since the computational of solving the model remains a concern, particularly when using \( p \)-th derivatives for \( p > 2 \), which may be impractical in many applications. Finally, extending our results for \( p=2 \) to the regularized Newton method \cite{mishchenko2023regularized,doikov2024gradient} is another interesting direction. 
%Finally, investigating potential applications in deep learning and structured nonconvex optimization could further illustrate the practical impact of our approach.  

 %The case of solving the nonconvex subproblem with \( h=0 \) for \( p = 3 \) has been studied in \cite{cartis2025cubic}.

%%%%%%%%%%%%%%%%%%%%%%%%%%%%%%%%%%%%%%
%%%%%%%%%%%%%%%%%%%%%%%%%%%%%%%%%%%%%%
\medskip
%\noindent \textbf{Data availability} Not applicable.
%\noindent\textbf{Conflicts of interest} The authors declare that they have no conflict of interest.
%%%%%%%%%%%%%%%%%%%%%%%%%%%%
%\bibliography{abbrevs,BiB}
%\bibliographystyle{jnsao}
%\bibliographystyle{plain}  % or IEEEtran, alpha, etc.
%\bibliography{BiB}  % Replace "mybib" with your actual .bib filename (without .bib extension)

%\begin{thebibliography}{99}
%\end{thebibliography}

%%%%%%%%%%%%%%%%%%%%%%%%%%%%
%%%%%%%%%%%%%%%%%%%%%%%%%%%%

%%%%%%%%%%%%%%%%%%%%%%%%%%%%
%%%%%%%%%%%%%%%%%%%%%%%%%%%%
\end{document}